\documentclass{article}
\usepackage{amsmath,amsthm,amsfonts,latexsym,amsopn,verbatim,amscd,amssymb}
\usepackage[colorlinks=true, linkcolor=red, filecolor=brown, citecolor=green, urlcolor=blue]{hyperref}

\theoremstyle{plain}
\newtheorem{theorem}{Theorem}[section]

\theoremstyle{definition}
\numberwithin{equation}{section}

\newcommand{\bnum}{\begin{enumerate}}
\newcommand{\enum}{\end{enumerate}}

\begin{document}
\title{Genus of commuting  graphs of some classes of finite rings }

\author{Walaa Nabil Taha  Fasfous  and Rajat Kanti Nath\footnote{Corresponding author} \\
Department of Mathematical Sciences, Tezpur University,\\ Napaam-784028, Sonitpur, Assam, India\\
}
 
\date{Email Addresses:  w.n.fasfous@gmail.com (W.N.T. Fasfous), rajatkantinath@yahoo.com (R.K. Nath)}
\maketitle

\smallskip

\noindent {\small{\textbf{Abstract:} }
In this paper, we compute the genus of commuting graphs of  non-commutative rings of order $p^4$, $p^5$, $p^2q$ and $p^3q$, where $p$ and $q$ are prime integers. We also characterize those  finite rings such that their commuting graphs are planar or toroidal.}

\bigskip



\section{Introduction} \label{S:intro}
The smallest non-negative integer $n$ such that a graph ${\mathcal{G}}$ can be embedded on the surface obtained by attaching $n$ handles to a sphere is called the genus of ${\mathcal{G}}$.  We write $\gamma({\mathcal{G}})$ to denote the genus of a graph ${\mathcal{G}}$. It is worth mentioning that 
\begin{equation}\label{genus1}
\gamma(K_n) = \left\lceil\frac{(n - 3)(n - 4)}{12}\right\rceil \text{ if } n \geq 3 \text{ (see \cite[Theorem 6-38]{wAT73})}.
\end{equation}
Also, if $\mathcal{G} = \overset{m}{\underset{i = 1}{\sqcup}}K_{n_i}$ then by    \cite[Corollary 2]{bhky62} we have
\begin{equation}\label{genus2}
\gamma({\mathcal{G}}) =  \overset{m}{\underset{i = 1}{\sum}}\gamma(K_{n_i}).
\end{equation}

\noindent A graph ${\mathcal{G}}$ is called planar or toroidal if $\gamma({\mathcal{G}}) = 0$ or $1$, respectively. The commuting graph $\Gamma_c (R)$, of a finite non-commutative ring $R$ with center $Z(R)$,  is an undirected graph whose vertex set is   $R\setminus Z(R)$ and two distinct vertices are adjacent if   they commute. The concept of $\Gamma_c (R)$  was introduced by Akbari, Ghandehari, Hadian and Mohammadian \cite{aghm04} in 2004. However, they studied $\Gamma_c (R)$  for semisimple rings.     
 In \cite{a08, m10} Abdollahi and Mohammadian respectively, considered commuting graphs of some matrix rings and  Omidi   and Vatandoost, in \cite{ov11}, initiated the study of commuting graphs of any finite non-commutative rings.  
  It is noteworthy that $\Gamma_c (R)$ is  less explored compared to the study of  commuting graphs of finite groups which was introduced by  Brauer and Fowler \cite{bF1955} in 1955. In \cite{walaa 0}, Dutta, Fasfous  and Nath have computed  genus of commuting  graphs of some classes of finite rings and characterized non-commutative rings of order $p^2$ and $p^3$ (for any prime) with unity such that their commuting graphs are planar or toroidal.  
In this paper, we consider non-commutative rings of order $p^4$, $p^5$, $p^2q$ and $p^3q$, where $p$ and $q$ are two primes, and  compute the genera  of their commuting graphs. We also characterize those rings such that their commuting graphs are planar or toroidal. It is worth mentioning that the structures of $\Gamma_c (R)$ for the above mentioned classes of rings have been described in \cite{vrb14,vrb16} and their spectral aspects have been explored in \cite{walaa 3}.

\section{ Main results} 
We begin with the following result.

\begin{theorem}\label{order-p4-genus}
Let $R$ be a non-commutative ring with unity and $|R| = p^4$.
\begin{enumerate}
\item Let $|Z(R)| = p$.
\begin{enumerate}
\item[(i)] If $p = 2$ then $\Gamma_c(R)$ is planar, toroidal or $\gamma(\Gamma_c(R))=2$.
\item[(ii)] If $p  \geq 3$ then $\gamma(\Gamma_c(R)) =  (p^2+p+1)\left\lceil\frac{(p^2- p -3)(p^2- p -4)}{12}\right
\rceil$ or \\
$l_1\left\lceil\frac{(p^2- p -3)(p^2- p -4)}{12}\right
\rceil + l_2\left\lceil\frac{(p^3- p -3)(p^3- p -4)}{12}\right\rceil$ for some positive integers $l_1, l_2$ such that $l_1 + l_2(p + 1) = p^2 + p + 1$; and hence $\Gamma_c(R)$ is neither  planar nor toroidal.
\end{enumerate}

\item Let $|Z(R)| = p^2$.
\begin{enumerate}
\item[(i)] $\Gamma_c(R)$ is planar if and only if $p = 2$.
\item[(ii)] If $p  \geq 3$ then
$\gamma(\Gamma_c(R)) = (p+1)\left\lceil\frac{(p^3-p^2 - 3)(p^3-p^2 - 4)}{12}\right\rceil$; and hence $\Gamma_c(R)$ is neither  planar nor toroidal.
\end{enumerate}
\end{enumerate}
\end{theorem}
\begin{proof}
(a) By  \cite[Theorem 2.5]{vrb16} we have $\Gamma_c(R) = (p^2 + p + 1)K_{p^2 - p}$ or  $l_1K_{p^2 - p} \sqcup  l_2K_{p^3 - p}$, where $l_1 + l_2(p + 1) = p^2 + p + 1$.

\noindent {\bf Case 1:} $\Gamma_c(R) = (p^2 + p + 1)K_{p^2 - p}$

By \eqref{genus2} we have $\gamma(\Gamma_c(R)) = (p^2 + p + 1)\gamma(K_{p^2 - p})$. If $p = 2$ then $p^2 - p = 2$ and so $\gamma(\Gamma_c(R)) = 7 \gamma(K_2) = 0$. Therefore,  $\Gamma_c(R)$ is planar. If $p  \geq 3$ then $p^2 - p \geq 6$. By \eqref{genus1} we have
\[
\gamma(\Gamma_c(R)) = (p^2 + p + 1)\left\lceil\frac{(p^2 - p - 3)(p^2 - p - 4)}{12}\right\rceil.
\]
Since $p  \geq 3$ we have  $\frac{(p^2 - p - 3)(p^2 - p - 4)}{12} \geq \frac{1}{2}$ and so $\gamma(\Gamma_c(R)) \geq 13$.
Thus $\Gamma_c(R)$ is neither  planar nor toroidal. 

\noindent {\bf Case 2:} $\Gamma_c(R) = l_1K_{p^2 - p} \sqcup  l_2K_{p^3 - p}$

By \eqref{genus2} we have $\gamma(\Gamma_c(R)) = l_1\gamma(K_{p^2 - p}) \sqcup  l_2\gamma(K_{p^3 - p})$. If $p = 2$ then $p^2 - p = 2$ and $p^3 - p = 6$. Also, $l_1 + 3l_2 = 7$ which gives $l_1 = 4$ and $l_2 = 1$ or $l_1 = 1$ and $l_2 = 2$. Therefore, $\gamma(\Gamma_c(R)) = 4\gamma(K_2) + \gamma(K_6) = 1$ or $\gamma(\Gamma_c(R)) = \gamma(K_2) + 2\gamma(K_6) = 2$. That is, $\Gamma_c(R)$ is toroidal or $\gamma(\Gamma_c(R))=2$. If $p  \geq 3$ then $p^2 - p \geq 6$ and $p^3 - p \geq 24$. By \eqref{genus1} we have
\[
\gamma(\Gamma_c(R)) = l_1\left\lceil\frac{(p^2- p -3)(p^2- p -4)}{12}\right
\rceil + l_2\left\lceil\frac{(p^3- p -3)(p^3- p -4)}{12}\right\rceil.
\]
Note that $\frac{(p^2- p -3)(p^2- p -4)}{12} \geq \frac{1}{2}$ and $\frac{(p^3- p -3)(p^3- p -4)}{12} \geq 35$. Therefore,  $\gamma(\Gamma_c(R)) \geq l_1 +  35l_2 > 36$. That is, $\Gamma_c(R)$ is neither planar nor toroidal. This completes the proof of part (a).

(b) By \cite[Theorem 2.5]{vrb16} we have $\Gamma_c(R) = (p + 1)K_{p^3 - p^2}$. Therefore, using  \eqref{genus2} we get $\gamma(\Gamma_c(R)) = (p + 1)\gamma(K_{p^3 - p^2})$. If $p = 2$ then $p^3 - p^2 = 4$. Therefore, $\gamma(\Gamma_c(R)) = 3\gamma(K_4) = 0$. That is, $\Gamma_c(R)$ is planar. If $p  \geq 3$ then $p^3 - p^2 \geq 18$ and so by \eqref{genus1} we have 
\[
\gamma(\Gamma_c(R)) = (p+1)\left\lceil\frac{(p^3-p^2 - 3)(p^3-p^2 - 4)}{12}\right\rceil.
\]
Note that $\frac{(p^3-p^2 - 3)(p^3-p^2 - 4)}{12} \geq \frac{35}{2}$ and so  $\gamma(\Gamma_c(R)) \geq 72$. Thus $\Gamma_c(R)$ is neither planar nor toroidal. This completes the proof of part (b). 
\end{proof}

\begin{theorem}\label{order-p5-genus}
Let $R$ be a non-commutative ring with unity where $|R| = p^5$ and $Z(R)$ not a field.
\begin{enumerate}
\item Let $|Z(R)| = p^2$.
\begin{enumerate}
\item[(i)] If $p = 2$ then $\Gamma_c(R)$ is planar, toroidal or $\gamma(\Gamma_c(R))=2$.
\item[(ii)] If $p  \geq 3$ then $\gamma(\Gamma_c(R)) =  (p^2+p+1)\left\lceil\frac{(p^3- p^2 -3)(p^3- p^2 -4)}{12}\right
\rceil$ or \\
$l_1\left\lceil\frac{(p^3- p^2 -3)(p^3- p^2 -4)}{12}\right
\rceil + l_2\left\lceil\frac{(p^3- p -3)(p^3- p -4)}{12}\right\rceil$ for some positive integers $l_1, l_2$ such that $l_1 + l_2(p + 1) = p^2 + p + 1$; and hence $\Gamma_c(R)$ is neither  planar nor toroidal.
\end{enumerate}

\item Let $|Z(R)| = p^3$. Then $\gamma(\Gamma_c(R)) = (p+1)\left\lceil\frac{(p^4-p^3 - 3)(p^4-p^3 - 4)}{12}\right\rceil$; and hence $\Gamma_c(R)$ is neither planar nor toroidal.
\end{enumerate}
\end{theorem}
\begin{proof}
(a) By  \cite[Theorem 2.5]{vrb16} we have $\Gamma_c(R) = (p^2 + p + 1)K_{p^3 - p^2}$ or  $l_1K_{p^3 - p^2} \sqcup  l_2K_{p^3 - p}$, where $l_1 + l_2(p + 1) = p^2 + p + 1$.

\noindent {\bf Case 1:} $\Gamma_c(R) = (p^2 + p + 1)K_{p^3 - p^2}$

By \eqref{genus2} we have $\gamma(\Gamma_c(R)) = (p^2 + p + 1)\gamma(K_{p^3 - p^2})$. If $p = 2$ then $p^3 - p^2 = 4$ and so $\gamma(\Gamma_c(R)) = 7 \gamma(K_4) = 0$. Therefore,  $\Gamma_c(R)$ is planar. If $p  \geq 3$ then $p^3 - p^2 \geq 18$. By \eqref{genus1} we have
\[
\gamma(\Gamma_c(R)) = (p^2 + p + 1)\left\lceil\frac{(p^3 - p^2- 3)(p^3 - p^2 - 4)}{12}\right\rceil.
\]
Since $p  \geq 3$ we have  $\frac{(p^3 - p^2 - 3)(p^3 - p^2 - 4)}{12} \geq \frac{35}{2}$ and so $\gamma(\Gamma_c(R)) \geq 234$.
Thus $\Gamma_c(R)$ is neither  planar nor toroidal. 

\noindent {\bf Case 2:} $\Gamma_c(R) = l_1K_{p^3 - p^2} \sqcup  l_2K_{p^3 - p}$

By \eqref{genus2} we have $\gamma(\Gamma_c(R)) = l_1\gamma(K_{p^3 - p^2}) \sqcup  l_2\gamma(K_{p^3 - p})$. If $p = 2$ then $p^3 - p^2 = 4$ and $p^3 - p = 6$. Also, $l_1 + 3l_2 = 7$ which gives $l_1 = 4$ and $l_2 = 1$ or $l_1 = 1$ and $l_2 = 2$. Therefore, $\gamma(\Gamma_c(R)) = 4\gamma(K_4) + \gamma(K_6) = 1$ or $\gamma(\Gamma_c(R)) = \gamma(K_4) + 2\gamma(K_6) = 2$. That is, $\Gamma_c(R)$ is toroidal or $\gamma(\Gamma_c(R))=2$. If $p  \geq 3$ then $p^3 - p^2 \geq 18$ and $p^3 - p \geq 24$. By \eqref{genus1} we have
\[
\gamma(\Gamma_c(R)) = l_1\left\lceil\frac{(p^3 - p^2 -3)(p^3 - p^2 -4)}{12}\right
\rceil + l_2\left\lceil\frac{(p^3- p -3)(p^3- p -4)}{12}\right\rceil.
\]
Note that $\frac{(p^3 - p^2 -3)(p^3 - p^2 -4)}{12} \geq \frac{35}{2}$ and $\frac{(p^3- p -3)(p^3- p -4)}{12} \geq 35$. Therefore,  $\gamma(\Gamma_c(R)) \geq 18l_1 +  35l_2 > 53$. That is, $\Gamma_c(R)$ is neither planar nor toroidal. This completes the proof of part (a).

(b) By \cite[Theorem 2.5]{vrb16} we have $\Gamma_c(R) = (p + 1)K_{p^4 - p^3}$. By \eqref{genus2} we have $\gamma(\Gamma_c(R)) = (p + 1)\gamma(K_{p^4 - p^3})$.
If $p  \geq 2$ then $p^4 - p^3 \geq 8$ and so by \eqref{genus1} we have 
\[
\gamma(\Gamma_c(R)) = (p+1)\left\lceil\frac{(p^4-p^3 - 3)(p^4-p^3 - 4)}{12}\right\rceil.
\]
Note that $\frac{(p^4-p^3 - 3)(p^4-p^3 - 4)}{12} \geq \frac{5}{3}$ and so  $\gamma(\Gamma_c(R)) \geq 6$. Thus $\Gamma_c(R)$ is neither planar nor toroidal. This completes the proof of part (b). 

\end{proof}

\begin{theorem}\label{order-p^2q-genus}
Let   $R$ be a non-commutative ring where $|R| =p^2q$ and $Z(R) = \{0\}$.
\begin{enumerate}
\item Let $t \in \{p, q, p^2, pq\}$ and $(t - 1) \mid (p^2q - 1)$ .
\begin{enumerate}
\item[(i)] $\Gamma_c(R)$ is planar if $t = p = q =2$; or $t = p^2=4$ and $q \geq 3$; or $t = p=2,3,5$ and $q \geq 3$; or $t = q=2,3,5$ and $p \geq 3$.
\item[(ii)] If $t=p \geq 7$ and $q \geq 3$; or $t=q \geq 7$ and $p \geq 3$; or $p \geq 3$,  $q \geq 3$ and $t=p^2$ or $pq$  then $\gamma(\Gamma_c(R)) =  \frac{p^2q-1}{t-1}\left\lceil\frac{(t-4)(t-5)}{12}\right
\rceil$; and hence $\Gamma_c(R)$ is neither  planar nor toroidal.
\end{enumerate}
\item Let $l_1(p - 1) + l_2(q - 1) + l_3(p^2 - 1) + l_4(pq - 1) = p^2q -1$ for some positive integers $l_1$, $l_2$, $l_3$ and $l_4$.
\begin{enumerate}
\item[(i)] If $p = 2 =q$ then $\Gamma_c(R)$ is planar. 
\item[(ii)] If $p=2$ and  $q=3$ then $\gamma(\Gamma_c(R))= l_4$; and hence $\Gamma_c(R)$ is not  planar but toroidal if and only if $l_4=1$.
\item[(iii)] If $p=2$ and  $q\geq5$ then $\gamma(\Gamma_c(R))= l_2\left\lceil\frac{(q - 4)(q - 5)}{12}\right \rceil +l_4\left\lceil\frac{(2q - 4)(2q - 5)}{12}\right \rceil$.
\item[(iv)] If $q=2$ and  $p=3$ then $\gamma(\Gamma_c(R))= 2l_3 +l_4$.
\item[(v)] If $q=2$ and  $p\geq5$ then $\gamma(\Gamma_c(R))= l_1\left\lceil\frac{(p - 4)(p - 5)}{12}\right \rceil +l_3\left\lceil\frac{(p^2 - 4)(p^2 - 5)}{12}\right \rceil+l_4\left\lceil\frac{(2p - 4)(2p - 5)}{12}\right \rceil$.
\item[(vi)] If $p=3=q$ then $\gamma(\Gamma_c(R))= 2(l_3 +l_4)$.
\item[(vii)] If $p=3$ and  $q\geq5$ then $\gamma(\Gamma_c(R))= l_2\left\lceil\frac{(q - 4)(q - 5)}{12}\right \rceil +2l_3 +l_4\left\lceil\frac{(3q - 4)(3q - 5)}{12}\right \rceil$.
\item[(viii)] If $p\geq5$ and  $q=3$ then $\gamma(\Gamma_c(R))= l_1\left\lceil\frac{(p - 4)(p - 5)}{12}\right \rceil +l_3\left\lceil\frac{(p^2 - 4)(p^2 - 5)}{12}\right \rceil+l_4\left\lceil\frac{(3p - 4)(3p - 5)}{12}\right \rceil$.
\item[(ix)] If $p\geq5$ and  $q\geq5$ then $\gamma(\Gamma_c(R))= l_1\left\lceil\frac{(p - 4)(p - 5)}{12}\right \rceil+ l_2\left\lceil\frac{(q - 4)(q - 5)}{12}\right \rceil +l_3\left\lceil\frac{(p^2 - 4)(p^2 - 5)}{12}\right \rceil+l_4\left\lceil\frac{(pq - 4)(pq - 5)}{12}\right \rceil$.\\
\end{enumerate}
It follows that $\Gamma_c(R)$ is neither  planar nor toroidal in all the cases $(iii) - (ix)$.
\end{enumerate}
\end{theorem}

\begin{proof}
(a) By \cite[Theorem 2.9]{vrb14} we have $\Gamma_c(R) =\frac{p^2q - 1}{t - 1}K_{t - 1}$.
By \eqref{genus2} we have $\gamma(\Gamma_c(R)) = \frac{p^2q - 1}{t - 1}\gamma(K_{t - 1})$.\\
\noindent {\bf Case 1:}  $p=2=q$ 

In this case $t=2$, since $(t-1)|(p^2q-1)=7$. Therefore,  $\gamma(\Gamma_c(R)) = 7 \gamma(K_1)=0$. 
That is,  $\Gamma_c(R)$ is planar.

\noindent {\bf Case 2:}  $p=2$ and $q \geq 3$ 

We have $(q-1)\nmid(4q-1)$ and $(2q-1)\nmid(4q-1)$. Therefore, $t\neq q$ and $t\neq pq=2q$.
If $t=p=2$ or $t=p^2=4$ and $q\geq 3$ then $t-1=1$ or $3$ and so $\gamma(\Gamma_c(R)) = 0$. That is,  $\Gamma_c(R)$ is planar.

\noindent {\bf Case 3:}  $p\geq 3$ and $q = 2$

We have $(p-1)\nmid(2p^2-1)$, $(p^2-1)\nmid(2p^2-1)$ and $(2p-1)\nmid(2p^2-1)$. Therefore, $t\neq p$, $p^2$ and $2p=pq$. If $t=q= 2$ then $t-1=1$ and so $\gamma(\Gamma_c(R)) = 0$. That is,  $\Gamma_c(R)$ is planar.

\noindent {\bf Case 4:}  $p\geq 3$ and $q \geq 3$

If $t=p=3$ or $5$ then $t-1=2$ or $4$ and so $\gamma(\Gamma_c(R)) = 0$. That is,  $\Gamma_c(R)$ is planar. 
If $t=p\geq 7$ then $t-1\geq 6$. Therefore, by  \eqref{genus1} we have  
\begin{center}
    $\gamma(\Gamma_c(R)) = \frac{p^2q - 1}{t - 1}\left\lceil\frac{(t- 4)(t-5)}{12}\right\rceil.$
\end{center}
Since  $\frac{(t- 4)(t-5)}{12} \geq \frac{1}{2}$ and $\frac{p^2q - 1}{t - 1} > 2$ we have  $\gamma(\Gamma_c(R)) >2$. That is, $\Gamma_c(R)$ is neither planar nor toroidal.
If $t=p^2$ then $t-1\geq 8$. Therefore, by \eqref{genus1} we have
\begin{center}
    $\gamma(\Gamma_c(R)) = \frac{p^2q - 1}{t - 1}\left\lceil\frac{(t- 4)(t-5)}{12}\right\rceil.$ 
\end{center}
Since  $\frac{(t- 4)(t-5)}{12} \geq \frac{5}{3}$ and $\frac{p^2q - 1}{t - 1} > 2$ we have  $\gamma(\Gamma_c(R)) >4$. That is, $\Gamma_c(R)$ is neither planar nor toroidal.
If $t=q=3$ or $5$ then $t-1=2$ or $4$ and so $\gamma(\Gamma_c(R)) = 0$. That is,  $\Gamma_c(R)$ is planar. 
If $t=q\geq 7$  then $t-1\geq 6$. Therefore, by  \eqref{genus1} we have  
\begin{center}
    $\gamma(\Gamma_c(R)) = \frac{p^2q - 1}{t - 1}\left\lceil\frac{(t- 4)(t-5)}{12}\right\rceil.$
\end{center}
Since  $\frac{(t- 4)(t-5)}{12} \geq \frac{1}{6}$ and $\frac{p^2q - 1}{t - 1} > 2$ we have  $\gamma(\Gamma_c(R)) >2$. That is, $\Gamma_c(R)$ is neither planar nor toroidal. If $t=pq$  then $t-1\geq 8$. Therefore, by  \eqref{genus1} we have  
\begin{center}
    $\gamma(\Gamma_c(R)) = \frac{p^2q - 1}{t - 1}\left\lceil\frac{(t- 4)(t-5)}{12}\right\rceil.$ 
\end{center}
Since  $\frac{(t- 4)(t-5)}{12} \geq \frac{5}{3}$ and $\frac{p^2q - 1}{t - 1} > 2$ we have  $\gamma(\Gamma_c(R)) >4$. That is, $\Gamma_c(R)$ is neither planar nor toroidal.

(b) By \cite[Theorem 2.9]{vrb14} we have $\Gamma_c(R)= l_1K_{p - 1}\sqcup l_2K_{q - 1} \sqcup l_3K_{p^2 - 1} \sqcup l_4K_{pq - 1}$.
By \eqref{genus2} we have $\gamma(\Gamma_c(R)) = l_1 \gamma(K_{p -1})  + l_2 \gamma(K_{q-1})  + l_3  \gamma(K_{p^2 -1})+ l_4 \gamma(K_{pq-1})$.\\
\noindent {\bf Case 1:} $p = 2 = q$ 

 In this case $\gamma(\Gamma_c(R)) = l_1 \gamma(K_1)  + l_2 \gamma(K_1)  + l_3  \gamma(K_3)+ l_4 \gamma(K_3)= 0$. Therefore, $\Gamma_c(R)$ is planar.
 
 \noindent {\bf Case 2:} $p = 2$ and $ q\geq3$ 
 
  In this case $\gamma(\Gamma_c(R)) = l_2 \gamma(K_{q-1})  + l_4 \gamma(K_{2q-1})$. If $q=3$ then $\gamma(\Gamma_c(R)) = l_2 \gamma(K_2)+l_4 \gamma(K_5) = l_4$. Therefore, $\Gamma_c(R)$ is not  planar since $l_4 \neq 0 $  and $\Gamma_c(R)$ is  toroidal if $l_4 =1$. If $q\geq 5$ then $q-1 \geq 4$ and $2q-1\geq 9$. Therefore, by  \eqref{genus1} we have
  \begin{center}
      $\gamma(\Gamma_c(R)) = l_2\left\lceil\frac{(q - 4)(q - 5)}{12}\right \rceil  +l_4\left\lceil\frac{(2q - 4)(2q - 5)}{12}\right \rceil $.
  \end{center}
  Since $\frac{(q-4)(q-5)}{12} \geq 0$ and $\frac{(2q-4)(2q-5)}{12} \geq \frac{5}{2}$ we have $\gamma(\Gamma_c(R)) \geq 3 l_4 \geq 3$. That is, $\Gamma_c(R)$ is neither planar nor toroidal.

\noindent {\bf Case 3:} $q=2$ and $p \geq 3$

In this case $\gamma(\Gamma_c(R)) = l_1 \gamma(K_{p-1}) +l_3 \gamma(K_{p^2-1})  + l_4 \gamma(K_{2p-1})$. If $p=3$ then $\gamma(\Gamma_c(R))= l_1 \gamma(K_{2}) +l_3 \gamma(K_{8})  + l_4 \gamma(K_{5}) = 2l_3 + l_4 \geq 3$. Therefore, $\Gamma_c(R)$ is neither planar nor toroidal. If $p\geq 5$ then $p-1 \geq 4$, $p^2-1\geq 24$ and $2p-1\geq 9$. Therefore, by  \eqref{genus1} we have 
 \begin{center}
      $\gamma(\Gamma_c(R)) = l_1\left\lceil\frac{(p - 4)(p - 5)}{12}\right \rceil + l_3\left\lceil\frac{(p^2 - 4)(p^2 - 5)}{12}\right \rceil +l_4\left\lceil\frac{(2p - 4)(2p - 5)}{12}\right \rceil $.
 \end{center}
Since $\frac{(p-4)(p-5)}{12} \geq 0$, $\frac{(p^2 - 4)(p^2 - 5)}{12} \geq 35$ and $\frac{(2p-4)(2p-5)}{12} \geq \frac{5}{2}$ we have $\gamma(\Gamma_c(R)) \geq 35l_3 + 3l_4 \geq 38$. That is, $\Gamma_c(R)$ is neither planar nor toroidal.

\noindent {\bf Case 4:} $p \geq 3$ and $q \geq 3$

If $p=3=q$ then $\gamma(\Gamma_c(R)) =l_3 \gamma(K_8)  + l_4 \gamma(K_8)= 2(l_3 + l_4) \geq 4$. Therefore, $\Gamma_c(R)$ is neither planar nor toroidal. If $p=3$ and $q\geq 5$ then 
\begin{center}
    $\gamma(\Gamma_c(R))= l_1 \gamma(K_2) + l_2 \gamma(K_{q-1})+ l_3 \gamma(K_8)  + l_4 \gamma(K_{3q-1}) $ 
\end{center}
\begin{center}
   $ ~~~~~~=l_2\left\lceil\frac{(q - 4)(q - 5)}{12}\right \rceil + 2l_3  + l_4 \left\lceil\frac{(3q - 4)(3q - 5)}{12}\right \rceil.$
\end{center}
Since $\frac{(q-4)(q-5)}{12} \geq 0$ and $\frac{(3q - 4)(3q - 5)}{12} \geq \frac{55}{6}$ we have  $\gamma(\Gamma_c(R)) \geq   2l_3 + 10l_4 \geq 12$. That is, $\Gamma_c(R)$ is neither planar nor toroidal. If $q=3$ and $p\geq 5$ then 
\begin{center}
    $\gamma(\Gamma_c(R))= l_1 \gamma(K_{p-1}) + l_2 \gamma(K_2)+ l_3 \gamma(K_{p^2-1})  + l_4 \gamma(K_{3p-1})$ 
\end{center}
\begin{center}
    $~~~~~~~~~~~~~~~~~~~~= l_1 \left\lceil\frac{(p - 4)(p - 5)}{12}\right \rceil + l_3 \left\lceil\frac{(p^2 - 4)(p^2- 5)}{12}\right \rceil + l_4 \left\lceil\frac{(3p - 4)(3p - 5)}{12}\right \rceil.$ 
\end{center}
Since $\frac{(p-4)(p-5)}{12} \geq 0$, $\frac{(p^2-4)(p^2-5)}{12} \geq 35$ and  $\frac{(3q - 4)(3q - 5)}{12} \geq \frac{55}{6}$ we have  $\gamma(\Gamma_c(R)) \geq   35l_3 + 10l_4 \geq 45$. That is, $\Gamma_c(R)$ is neither planar nor toroidal. If $p\geq 5$  and $q\geq 5$ then\\ $\gamma(\Gamma_c(R))= l_1 \left\lceil\frac{(p - 4)(p - 5)}{12}\right \rceil + l_2 \left\lceil\frac{(q - 4)(q - 5)}{12}\right \rceil + l_3 \left\lceil\frac{(p^2 - 4)(p^2- 5)}{12}\right \rceil + l_4 \left\lceil\frac{(pq - 4)(pq - 5)}{12}\right \rceil.$\\ 
Since $\frac{(p-4)(p-5)}{12} \geq 0$, $\frac{(q-4)(q-5)}{12} \geq 0$, $\frac{(p^2-4)(p^2-5)}{12} \geq 35$ and  $\frac{(pq - 4)(pq - 5)}{12} \geq 35$ we have $\gamma(\Gamma_c(R)) \geq   35l_3 + 35l_4 \geq 70$. That is, $\Gamma_c(R)$ is neither planar nor toroidal.
\end{proof}

\begin{theorem}\label{order-p^3 q-1 -genus}
Let $R$ be a non-commutative ring with unity where $ |R|=p^3q$ and $|Z(R)| = pq$. 
\begin{enumerate}
\item If $p = 2 =q$ then $\Gamma_c(R)$ is planar. 
\item If $p=2$ and  $q\geq3$ then $\gamma(\Gamma_c(R))= 3\left\lceil\frac{(2q - 3)(2q - 4)}{12}\right \rceil$.
\item If $q=2$ and  $p\geq3$ then $\gamma(\Gamma_c(R))= (p+1)\left\lceil\frac{(2p^2-2p - 3)(2p^2-2p -4)}{12}\right \rceil$.
\item If $p\geq3$ and  $q\geq3$ then $\gamma(\Gamma_c(R))= (p+1) \left\lceil\frac{(p^2q-pq-3)(p^2q-pq-4)}{12}\right \rceil$.\\
\end{enumerate}
It follows that $\Gamma_c(R)$ is neither  planar nor toroidal in all the cases $(b) - (d)$.
  \end{theorem}
  \begin{proof}
    By \cite[Theorem 2.12]{vrb14} we have $\Gamma_c(R) =(p+1)K_{p^2q- pq}$.
By \eqref{genus2} we have $\gamma(\Gamma_c(R)) = (p+1)\gamma(K_{p^2q- pq})$.

\noindent {\bf Case 1:} $p = 2 = q$ 

 In this case $\gamma(\Gamma_c(R)) = 3 \gamma(K_4) = 0$. Therefore, $\Gamma_c(R)$ is planar.
 
 \noindent {\bf Case 2:} $p = 2$ and $ q\geq3$ 
 
  In this case we have $\gamma(\Gamma_c(R)) = 3 \gamma(K_{2q})$ and $2q \geq 6$.
  Therefore, by  \eqref{genus1} we have
 \begin{center}
      $\gamma(\Gamma_c(R)) = 3\left\lceil\frac{(2q - 3)(2q - 4)}{12}\right \rceil$.
 \end{center}
Since $\frac{(2q-3)(2q-4)}{12} \geq \frac{1}{2}$ we have  $\gamma(\Gamma_c(R)) \geq 3$. That is, $\Gamma_c(R)$ is neither planar nor toroidal.

\noindent {\bf Case 3:} $q=2$ and $p \geq 3$

In this case we have $\gamma(\Gamma_c(R)) =  (p+1)\gamma(K_{2p^2-2p})$ and $2p^2-2p \geq 12$.
Therefore, by  \eqref{genus1} we have
 \begin{center}
     $\gamma(\Gamma_c(R)) = (p+1)\left\lceil\frac{(2p^2-2p - 3)(2p^2 -2p-4 )}{12}\right \rceil$.
 \end{center}
Since $\frac{(2p^2-2p - 3)(2p^2 -2p-4 )}{12} \geq 6$ and $p+1 \geq 4$ we have $\gamma(\Gamma_c(R)) \geq 24$. That is, $\Gamma_c(R)$ is neither planar nor toroidal.

\noindent {\bf Case 4:} $p \geq 3$ and $q \geq 3$

We have $p^2q -pq \geq 18$.
Therefore, by  \eqref{genus1} we have  
\begin{center}
    $\gamma(\Gamma_c(R)) = (p+1)\left\lceil\frac{(p^2q-pq - 3)(p^2q -pq-4 )}{12}\right \rceil$.
\end{center}
Since $\frac{(p^2q-pq - 3)(p^2q -pq-4 )}{12}\geq \frac{35}{2}$ and $p+1 \geq 4$ we have $\gamma(\Gamma_c(R)) \geq   72$. That is, $\Gamma_c(R)$ is neither planar nor toroidal.
 \end{proof}
   
   \begin{theorem}\label{order-p^3q-2-genus}
Let   $R$ be a non-commutative ring where $|R| =p^3q$ and $|Z(R)| = p^2$.

\begin{enumerate}
\item Let  $(q - 1)\mid (pq - 1)$.
\begin{enumerate}
\item[(i)] If $p = 2 =q$ then $\Gamma_c(R)$ is planar.
\item[(ii)] If $q=2$ and  $p\geq3$ then $\gamma(\Gamma_c(R))= 2p-1\left\lceil\frac{(p^2-3)(p^2-4)}{12}\right \rceil$.
\item[(iii)] If $p\geq3$ and  $q\geq3$ then $\gamma(\Gamma_c(R))= \frac{pq-1}{q-1}\left\lceil\frac{(p^2q-p^2-3)(p^2q-p^2-4)}{12}\right \rceil$.\\
\end{enumerate}
It follows that $\Gamma_c(R)$ is neither  planar nor toroidal in the cases $(ii)$ and $(iii)$.
\item Let  $(p - 1)\mid (pq - 1)$.
\begin{enumerate}
\item[(i)] If $p = 2$ and $ q\geq 2$ then $\Gamma_c(R)$ is planar.
\item[(ii)] If $p\geq3$ and  $q\geq3$ then $\gamma(\Gamma_c(R))= \frac{pq-1}{p-1}\left\lceil\frac{(p^3-p^2-3)(p^3-p^2-4)}{12}\right \rceil$; and hence $\Gamma_c(R)$ is neither  planar nor toroidal. 
\end{enumerate}
\item Let $l_1(p - 1) + l_2(q - 1) = pq -1$, where $l_1$ and $l_2$ are positive integers.
\begin{enumerate}
\item[(i)] If $p = 2 =q$ then $\Gamma_c(R)$ is planar.
\item[(ii)] If $p=2$ and  $q\geq3$ then $\gamma(\Gamma_c(R))= l_2\left\lceil\frac{(4q - 7)(4q - 8)}{12}\right \rceil$.
\item[(iii)] If $q=2$ and  $p\geq3$ then $\gamma(\Gamma_c(R))= l_1\left\lceil\frac{(p^3-p^2 - 3)(p^3-p^2-4)}{12}\right \rceil +l_2\left\lceil\frac{(p^2 - 3)(p^2 -4)}{12}\right \rceil$.
\item[(iv)] If $p\geq3$ and  $q\geq3$ then\\ $\gamma(\Gamma_c(R))= l_1\left\lceil\frac{(p^3-p^2 - 3)(p^3-p^2-4)}{12}\right \rceil +l_2\left\lceil\frac{(p^2q-p^2 - 3)(p^2q-p^2 -4)}{12}\right \rceil$.
\end{enumerate}
It follows that $\Gamma_c(R)$ is neither  planar nor toroidal in  all the cases $(ii) - (iv)$.
\end{enumerate}
\end{theorem}
\begin{proof}
(a) By \cite[Theorem 2.12]{vrb14} we have $\Gamma_c(R) =\frac{pq - 1}{q - 1}K_{p^2q- p^2}$.
By \eqref{genus2} we have $\gamma(\Gamma_c(R)) = \frac{pq - 1}{q - 1}\gamma(K_{p^2q- p^2})$. We shall complete the proof by considering the following cases. Note that the case $p=2$ and $q\geq3$ does not arise since $(q-1)\nmid(2q-1)$.

\noindent {\bf Case 1:}  $p=2=q$ 

In this case we have  $\Gamma_c(R) = 3 K_4$ and so $\gamma(\Gamma_c(R)) = 0$. That is,  $\Gamma_c(R)$ is planar.

\noindent {\bf Case 2:} $q=2$ and $p \geq 3$

In this case $\gamma(\Gamma_c(R)) =  (2p-1)\gamma(K_{p^2})$ where $p^2 \geq 9$. 
Therefore, by  \eqref{genus1} we have  
\begin{center}
  $\gamma(\Gamma_c(R)) = (2p-1)\left\lceil\frac{(p^2 - 3)(p^2-4 )}{12}\right \rceil$.  
\end{center}
Since $\frac{(p^2 - 3)(p^2-4 )}{12} \geq \frac{5}{2}$ and $2p-1 \geq 5$ we have  $\gamma(\Gamma_c(R)) > 15 $. That is, $\Gamma_c(R)$ is neither planar nor toroidal.

\noindent {\bf Case 3:} $p \geq 3$ and $q \geq 3$

We have $p^2q-p^2 \geq 18 $.
Therefore, by  \eqref{genus1} we have  
\begin{center}
    $\gamma(\Gamma_c(R)) = \frac{pq-1}{q-1}\left\lceil\frac{(p^2q-p^2 - 3)(p^2q-p^2-4 )}{12}\right \rceil$.
\end{center}
Since $\frac{(p^2q-p^2 - 3)(p^2q -p^2-4 )}{12} \geq \frac{35}{2}$ and $\frac{pq-1}{q-1} > 2$ we have  $\gamma(\Gamma_c(R)) > 36$. That is, $\Gamma_c(R)$ is neither planar nor toroidal.

(b) By \cite[Theorem 2.12]{vrb14} we have $\Gamma_c(R) =\frac{pq - 1}{p - 1}K_{p^3- p^2}$.
By \eqref{genus2} we have $\gamma(\Gamma_c(R)) = \frac{pq - 1}{p - 1}\gamma(K_{p^3- p^2})$. We shall complete the proof by considering the following cases. Note that the case $q=2$ and $p\geq3$ does not arise since $(p-1)\nmid(2p-1)$.



\noindent {\bf Case 1:} $p=2$ and $q \geq 2$

In this case we have $\gamma(\Gamma_c(R)) = (2q-1)\gamma(K_{4})=0$. That is,  $\Gamma_c(R)$ is planar. 



\noindent {\bf Case 2:} $p \geq 3$ and $q \geq 3$

We have $p^3-p^2 \geq 18$.
Therefore, by  \eqref{genus1} we have  
\begin{center}
    $\gamma(\Gamma_c(R)) = \frac{pq-1}{p-1}\left\lceil\frac{(p^3-p^2 - 3)(p^3-p^2-4 )}{12}\right \rceil$.
\end{center}
Since $\frac{(p^3-p^2 - 3)(p^3-p^2-4 )}{12} \geq \frac{35}{2}$ and $\frac{pq-1}{p-1} > 2$ we have  $\gamma(\Gamma_c(R)) >  36$. That is, $\Gamma_c(R)$ is neither planar nor toroidal.

(c) By \cite[Theorem 2.12]{vrb14} we have $
\Gamma_c(R)= l_1K_{p^3- p^2}\sqcup l_2K_{p^2q- p^2}$.
By \eqref{genus2} we have $\gamma(\Gamma_c(R)) = l_1\gamma(K_{p^3- p^2}) +  l_2\gamma(K_{p^2q- p^2})$.

\noindent {\bf Case 1:} $p = 2 = q$ 

 In this case $\gamma(\Gamma_c(R)) = l_1 \gamma(K_4)  + l_2 \gamma(K_4) = 0$. Therefore, $\Gamma_c(R)$ is planar.
 
 \noindent {\bf Case 2:} $p = 2$ and $ q\geq3$ 
 
  In this case $\gamma(\Gamma_c(R)) = l_1 \gamma(K_{4})  + l_2 \gamma(K_{4q-4})=l_2\gamma(K_{4q-4})$. We have $4q-4 \geq 8$.
  Therefore, by  \eqref{genus1} we have \\
  \begin{center}
      $\gamma(\Gamma_c(R)) = l_2\left\lceil\frac{(4q-7)(4q-8)}{12}\right \rceil$.
  \end{center}
Since $\frac{(4q-7)(4q-8)}{12} \geq \frac{5}{3}$ we have $\gamma(\Gamma_c(R)) \geq 2l_2 \geq 2$. That is, $\Gamma_c(R)$ is neither planar nor toroidal.

\noindent {\bf Case 3:} $q=2$ and $p \geq 3$

In this case $\gamma(\Gamma_c(R)) = l_1 \gamma(K_{p^3-p^2})  + l_2 \gamma(K_{p^2})$. 
Since $p^3-p^2 \geq 18$ and $p^2 \geq 9$, by  \eqref{genus1} we have\\
\begin{center}
     $\gamma(\Gamma_c(R)) = l_1\left\lceil\frac{(p^3-p^2-3)(p^3-p^2-4)}{12}\right \rceil + l_2\left\lceil\frac{(p^2 - 3)(p^2 - 4)}{12}\right \rceil$.
\end{center}
Since $\frac{(p^3-p^2-3)(p^3-p^2-4)}{12} \geq \frac{35}{2}$ and $\frac{(p^2 - 3)(p^2 - 4)}{12} \geq \frac{5}{2}$ we have  $\gamma(\Gamma_c(R)) \geq 18l_1 + 3l_2 \geq 21$. That is, $\Gamma_c(R)$ is neither planar nor toroidal.

\noindent {\bf Case 4:} $p \geq 3$ and $q \geq 3$

We have $p^3-p^2 \geq 18$ and $p^2q -p^2 \geq 18$.
Therefore, by  \eqref{genus1} we have 
\begin{center}
      $\gamma(\Gamma_c(R))= l_1 \left\lceil\frac{(p^3-p^2 - 3)(p^3-p^2-4)}{12}\right \rceil + l_2 \left\lceil\frac{(p^2q- p^2-3 )(p^2q- p^2-4)}{12}\right\rceil.$ 
 \end{center}
Since $\frac{(p^3-p^2 - 3)(p^3-p^2-4)}{12} \geq \frac{35}{2}$ and $\frac{(p^2q- p^2-3 )(p^2q- p^2-4)}{12} \geq \frac{35}{2}$ we have $\gamma(\Gamma_c(R)) \geq   18l_1 + 18l_2 \geq 36$. That is, $\Gamma_c(R)$ is neither planar nor toroidal.

\end{proof}
 

\begin{thebibliography}{20}

\bibitem{a08}
A.~Abdollahi,  \emph{ Commuting graphs of full matrix rings over finite fields}, Linear Algebra Appl. {\bf 428} (2008),  2947--2954.


\bibitem{aghm04}
S.~Akbari, M.~Ghandehari, M.~Hadian and A.~Mohammadian, \emph{ On commuting graphs of semisimple rings}, Linear Algebra Appl.
{\bf 390} (2004), 345--355.


\bibitem{bhky62}
J.~Battle, F.~Harary, Y.~Kodama and J. W. T.~Youngs,
\emph{ Additivity of the genus of a graph},  Bull. Amer. Math. Soc. 
{\bf 68} (1962), 565--568. 

\bibitem{bF1955}
R. Brauer and K. A. Fowler, {\em On groups of even order},  Ann.  Math.  {\bf 62}(2) (1955), 565--583.


\bibitem{walaa 0} 
J.~Dutta, W.~Fasfous and R.~Nath,  \emph{ Spectrum and genus of commuting  graphs of some classes of finite rings}, Acta Comment. Univ. Tartu. Math. {\bf 23}(1) (2019),  5--12.

\bibitem{walaa 3} 
W.~Fasfous, R.~Nath, and R.~Sharafdini  \emph{ Various spectra and  energies of commuting  graphs of finite rings}, Hacet. J. Math.  Stat.  {\bf 49}(6) (2020), 1915--1925, 2020.


\bibitem{m10} 
A. Mohammadian, \emph{ On commuting graphs of finite matrix rings},  Comm. Algebra {\bf 38} (2010),  988--994. 


\bibitem{ov11}
G. R.~Omidi   and E.~Vatandoost, \emph{  On the commuting graph of rings}, J. Algebra Appl. 
{\bf 10} (2011), 521--527.

\bibitem{vrb14}
E.~Vatandoost, F.~Ramezani   and A.~Bahraini,
\emph{  On the commuting graph of non-commutative rings of order $p^nq$},  J. Linear Topological Algebra 
{\bf 3} (2014),   1--6.

\bibitem{vrb16}
E.~Vatandoost and  F.~Ramezani,
\emph{On the commuting graph of some non-commutative rings with unity},  J. Linear Topological Algebra 
{\bf 5}(4) (2016), 289--294.


\bibitem{wAT73}
A. T.~White, \emph{ Graphs, Groups and Surfaces},
 North-Holland Mathematics Studies {\bf 8}, 
American Elsevier Publishing Co. New York, 1973.

 \end{thebibliography}
\end{document}